\documentclass[12pt]{article}

\usepackage{amsthm}
\usepackage{fullpage}
\usepackage{amssymb, amsmath}
\usepackage{hyperref}
\usepackage{graphicx}
\usepackage{titlesec}
\usepackage{enumitem}
\usepackage{thmtools}
\usepackage{thm-restate}
\newtheorem{thm}{Theorem}
\newtheorem{lem}[thm]{Lemma}
\newtheorem{cor}[thm]{Corollary}
\newtheorem{conj}[thm]{Conjecture}
\newtheorem{obs}[thm]{Observation}
\newtheorem*{claim}{Claim}

\title{Orientation-based edge-colorings and linear arboricity of multigraphs}
\author{Ronen Wdowinski\thanks{Department of Combinatorics and Optimization, University of Waterloo, Waterloo, ON, Canada. Email: ronen.wdowinski@uwaterloo.ca.}}
\date{\today}

\begin{document}
\maketitle
\begin{abstract}
    The Goldberg-Seymour Conjecture for $f$-colorings states that the $f$-chromatic index of a loopless multigraph is essentially determined by either a maximum degree or a maximum density parameter. We introduce an oriented version of $f$-colorings, where now each color class of the edge-coloring is required to be orientable in such a way that every vertex $v$ has indegree and outdegree at most some specified values $g(v)$ and $h(v)$. We prove that the associated $(g,h)$-oriented chromatic index satisfies a Goldberg-Seymour formula. We then present simple applications of this result to variations of $f$-colorings. In particular, we show that the Linear Arboricity Conjecture holds for $k$-degenerate loopless multigraphs when the maximum degree is at least $4k-2$, improving a bound recently announced by Chen, Hao, and Yu for simple graphs. Finally, we demonstrate that the $(g,h)$-oriented chromatic index is always equal to its list coloring analogue.
\end{abstract}

\section{Introduction}
For basic terminology on multigraphs, simple graphs, and oriented graphs, see Diestel \cite{Di}.

\subsection{f-colorings}

Given a loopless multigraph $G$ and a function $f : V(G) \rightarrow \mathbb{N} \backslash \{0\}$, an \textit{$f$-coloring} of $G$ is an assignment of a color to each edge of $G$ such that each color class is a subgraph of $G$ in which every vertex $v \in V(G)$ has degree $d(v)$ at most $f(v)$. We will refer to such subgraphs as \textit{degree-f subgraphs}. If every vertex $v$ has degree exactly $f(v)$, such subgraphs are more commonly known as \textit{$f$-factors}. The \textit{$f$-chromatic index} $\chi_f'(G)$ is the minimum number of colors needed in an $f$-coloring of $G$. These definitions can be extended to multigraphs with loops as long as $f(v) \ge 2$ for every vertex $v$ with a loop. For notation, if $f(v) = c$ for all $v \in V(G)$ where $c$ is a constant, we will write $f = c$. For a vertex subset $S \subseteq V(G)$, we will write $f(S) = \sum_{v \in S} f(v)$.

The notion of an $f$-coloring was introduced by Hakimi and Kariv \cite{HaKa} as a way to generalize the case $f=1$ of a \textit{proper edge-coloring}, where $\chi_1'(G) = \chi'(G)$ is known as the \textit{chromatic index} of $G$. In general it is NP-hard to determine $\chi'(G)$, even if $G$ is a simple graph \cite{Ho}, so we cannot expect there to be an efficient method to determine the $f$-chromatic index $\chi_f'(G)$ in general. However, the Goldberg-Seymour Conjecture for $f$-colorings asserts that we should always be able to determine $\chi_f'(G)$ to within an additive error of $1$.

First, one easy lower bound for the $f$-chromatic index is $\chi_f'(G) \ge \Delta_f(G)$ where 
\begin{align*}
    \Delta_f(G) = \max_{v \in V(G)} \left\lceil \frac{d(v)}{f(v)} \right\rceil
\end{align*}
is a weighted maximum degree; this is because every vertex $v$ has degree at most $f(v)$ in every degree-$f$ subgraph of $G$. Another easy lower bound is $\chi_f'(G) \ge \mathcal{W}_f(G)$ where
\begin{align*}
    \mathcal{W}_f(G) = \max_{S \subseteq V(G), |S| \ge 2} \left\lceil \frac{e(S)}{\lfloor f(S)/2 \rfloor} \right\rceil
\end{align*}
is a weighted maximum density (with $e(S)$ denoting the number of edges in the induced subgraph $G[S]$); this is because every degree-$f$ subgraph on a vertex subset $S \subseteq V(G)$ has at most $\left\lfloor \frac{1}{2} \sum_{v \in S} f(v) \right\rfloor = \lfloor f(S)/2 \rfloor$ edges. Inspired by the above famous conjecture due to Gupta \cite{Gu4}, Goldberg \cite{Go}, Andersen \cite{An}, and Seymour \cite{Se2} for the case $f=1$, Nakano, Nishizeki, and Saito \cite{NaNiSa} conjectured that the $f$-chromatic index $\chi_f'(G)$ should be essentially determined by either $\Delta_f(G)$ or $\mathcal{W}_f(G)$.
\begin{conj} \label{Goldberg-Seymour-f}
For every loopless multigraph $G$ and function $f : V(G) \rightarrow \mathbb{N} \backslash \{0\}$, we have
\begin{align*}
    \chi_f'(G) \le \max \left\{\Delta_f(G) + 1, \mathcal{W}_f(G)\right\}
\end{align*}
\end{conj}
This is the Goldberg-Seymour Conjecture for $f$-colorings. Hakimi and Kariv \cite{HaKa} proved that $\chi_f'(G) \le \Delta_f(G)+1$ if $G$ is a simple graph, generalizing the classical upper bound of Vizing for the chromatic index \cite{Vi}. They also proved a generalization of Shannon's upper bound \cite{Sh}. Nakano, Nishizeki, and Saito \cite{NaNiSa} themselves approximated Conjecture \ref{Goldberg-Seymour-f} to within a factor of $9/8$. Stiebitz, Scheide, Toft, and Favrholdt \cite{St} have proven a fractional version of Conjecture \ref{Goldberg-Seymour-f}. Much better approximations have been obtained for the case $f=1$ of the chromatic index using the tool Tashkinov trees (see \cite{ChGaKiPoSh, ChYuZa, Sc, Ta}). A proof of the case $f=1$ has even recently been announced by Chen, Jing, and Zang \cite{ChJiZa}, although it awaits verification. See \cite{St} for an overview of results relevant for Conjecture \ref{Goldberg-Seymour-f}.

\subsection{Oriented-colorings}

In this paper we introduce an oriented version of $f$-colorings, from which a result analogous to the Goldberg-Seymour Conjecture \ref{Goldberg-Seymour-f} for $f$-colorings can easily be proven. Given two functions $g, h : V(G) \rightarrow \mathbb{N} \cup \{\infty\}$, we say that a multigraph $G$ (possibly with loops) is \textit{$(g,h)$-orientable} if it has an orientation $D$ such that every vertex $v$ of $D$ has indegree $d_D^-(v)$ at most $g(v)$ and outdegree $d_D^+(v)$ at most $h(v)$. We call such an orientation $D$ a \textit{$(g,h)$-orientation}. A characterization of $(g,h)$-orientable multigraphs has already been given in a relatively little-known paper of Entringer and Tolman \cite{EnTo}, but these kinds of multigraphs have not before been the subject of edge-colorings except in a few special cases. A \textit{$(g,h)$-oriented-coloring} of $G$ is an assignment of a color to each edge of $G$ such that each color class is a $(g,h)$-orientable subgraph of $G$. We define the \textit{$(g,h)$-oriented chromatic index} $\chi'_{(g,h)}(G)$ to be the minimum number of colors needed in a $(g,h)$-oriented-coloring of $G$. 

To ensure that $\chi'_{(g,h)}(G)$ exists, we will assume that the following conditions hold:
\begin{itemize}
    \item $g(v) + h(v) \ge 1$ for all $v \in V(G)$,
    \item $g(u) + g(v) \ge 1$ for all adjacent $u,v \in V(G)$, and
    \item $h(u) + h(v) \ge 1$ for all adjacent $u,v \in V(G)$,
\end{itemize}
where $v$ is considered adjacent to itself if there is a loop at $v$. The first condition ensures that every vertex $u$ accepts either the head or the tail of any arc incident to it, the second condition ensures that every neighbor $v$ of $u$ accepts the head of an arc if $u$ does not, and the third condition ensures that the neighbors of $u$ accept the tail of an arc if $u$ does not. We call $(g,h)$ a \textit{valid pair of functions} for $G$ if they satisfy these three conditions.

Let $G$ be a multigraph and let $(g,h)$ be a valid pair of functions for $G$. As with the $f$-chromatic index $\chi_f'(G)$, there is a maximum degree and a maximum density lower bound for $\chi'_{(g,h)}(G)$. Define the weighted maximum degree parameter
\begin{align*}
    \Delta_{(g,h)}(G) = \max_{v \in V(G)} \left\lceil \frac{d(v)}{g(v) + h(v)} \right\rceil
\end{align*}
and the weighted maximum density parameter
\begin{align*}
    \mathcal{W}_{(g,h)}(G) = \max_{\substack{S \subseteq V(G),\\ g(S), h(S) \ge 1}} \left\lceil \frac{e(S)}{\min\left\{g(S), h(S)\right\}} \right\rceil.
\end{align*}
(To deal with denominators of $\infty$, we may replace outputs of $\infty$ in $g$ and $h$ by $\Delta(G)$.) Observe that $\chi'_{(g,h)}(G) \ge \Delta_{(g,h)}(G)$: in every $(g,h)$-orientable subgraph of $G$, every vertex $v$ has degree at most $g(v) + h(v)$. Also observe that $\chi'_{(g,h)}(G) \ge \mathcal{W}_{(g,h)}(G)$: for every $S \subseteq V(G)$, every $(g,h)$-orientable subgraph of $G[S]$ has at most $g(S)$ edges (by summing indegrees) and at most $h(S)$ edges (by summing outdegrees), so at most $\min \{g(S), h(S)\}$ edges. We will prove the following Goldberg-Seymour result for the $(g,h)$-oriented chromatic index.

\begin{restatable}{thm}{orientedgoldbergseymour}\label{oriented-Goldberg-Seymour}
For every multigraph $G$ and valid pair of functions $(g,h)$, we have
\begin{align*}
    \chi'_{(g,h)}(G) = \max \left\{\Delta_{(g,h)}(G), \mathcal{W}_{(g,h)}(G)\right\}.
\end{align*}
\end{restatable}

The proof will be elementary, following straightforwardly from known results. For much of the remainder of the paper, our focus will be on elementary applications of Theorem \ref{oriented-Goldberg-Seymour}. In Section \ref{on-pseudoarboricity}, we will use the theorem to solve the $f$-coloring problem for pseudoforests (Theorem \ref{f-pseudoarboricity}). In Section \ref{on-arboricity}, we will use this result on pseudoforests to approximate the $f$-coloring problem for forests (Theorem \ref{f-arboricity}). Specializing to $f = 2$, our result will give an upper bound on the linear arboricity of a loopless multigraph. In particular, in Section \ref{on-linear-arboricity} we will show that the unsolved Linear Arboricity Conjecture holds for $k$-degenerate loopless multigraphs $G$ when $\Delta(G) \ge 4k-2$ (Corollary \ref{linear-arboricity-corollary}), improving the bound $\Delta(G) \ge 2k^2 - 2k$ recently announced by Chen, Hao, and Yu \cite{ChHa} for simple graphs $G$. Then specializing to $f = t$, we will improve an old bound of Caro and Roditty \cite{CaRo} on the degree-$t$ arboricity $a_t(G)$ of a simple graph $G$ at least when $\Delta(G)$ is large enough (Corollary \ref{t-arboricity}), and our result will also apply more generally to when $G$ is a loopless multigraph. In Section \ref{on-f-chromatic-index}, we will note results on the $f$-chromatic index that we can obtain using Theorem \ref{oriented-Goldberg-Seymour}, including an approximation of the Goldberg-Seymour Conjecture \ref{Goldberg-Seymour-f} for $f$-colorings (Theorem \ref{odd-cor}).

Finally, in Section \ref{list}, using Galvin's list coloring theorem on bipartite multigraphs \cite{Ga} we will easily demonstrate that $\chi'_{(g,h)}(G)$ and its list coloring analogue $\chi'_{(g,h),\ell}(G)$ are always equal (Theorem \ref{list-version}), similar to the assertion of the currently unsolved List Coloring Conjecture for proper edge-colorings. We will then indicate direct consequences of this result for the list coloring analogues of the other edge-coloring parameters studied above.

\section{Proof of Theorem \ref{oriented-Goldberg-Seymour}} \label{proof}

Our proof of Theorem \ref{oriented-Goldberg-Seymour} will be a straightforward consequence of Lemma \ref{orientation-theorem} and Lemma \ref{coloring-lemma} below. It is inspired by the well-known bipartite matching proof of Petersen's 2-Factor Theorem (that every $2k$-regular multigraph can be decomposed into $k$ $2$-factors, see \cite{Di}). We will need the following classical edge-coloring theorem of K\"onig \cite{Ko}.

\begin{thm}[K\"onig]\label{Konig}
For every bipartite multigraph $G$, we have $\chi'(G) = \Delta(G)$.
\end{thm}

We will also need the following less well-known result characterizing $(g,h)$-orientable multigraphs, due to Entringer and Tolman \cite{EnTo} (their Corollary 1(v)). It is a generalization of Hakimi's more well-known characterization of $(g,\infty)$-orientable multigraphs \cite{Ha1}.

\begin{restatable}[Entringer, Tolman]{lem}{orientationtheorem} \label{orientation-theorem}
For a multigraph $G$ and a pair of functions $(g,h)$, $G$ is $(g,h)$-orientable if and only if
\begin{enumerate}[label={(\arabic*)}]
    \item $d(v) \le g(v) + h(v)$ for all $v \in V(G)$, and
    \item $e(S) \le \min \{ g(S), h(S) \}$ for all $S \subseteq V(G)$.
\end{enumerate}
\end{restatable}

Conditions (1) and (2) in Lemma \ref{orientation-theorem} are clearly necessary, as explained in the introduction. The main theorem of Entringer and Tolman actually also includes lower bounds for the indegree and outdegree constraints. For completeness, we reprove the special case of Lemma \ref{orientation-theorem} in Appendix \ref{appendix}. Our path reversal proof resembles the well-known alternating path proof of K\"onig's Theorem \ref{Konig} (see \cite{Di}), although it is independent of it. 

Now, we observe that Lemma \ref{orientation-theorem} implies Theorem \ref{oriented-Goldberg-Seymour} when $\chi'_{(g,h)}(G)=1$. To prove Theorem \ref{oriented-Goldberg-Seymour} for larger values of $\chi'_{(g,h)}(G)$, we prove the following consequence of K\"onig's Theorem \ref{Konig}.

\begin{lem}\label{coloring-lemma}
For a multigraph $G$, a valid pair of functions $(g,h)$, and an integer $k \ge 1$, we have $\chi_{(g,h)}'(G) \le k$ if and only if $G$ is $(kg, kh)$-orientable.
\end{lem}

\begin{proof}
Suppose first that $\chi_{(g,h)}'(G) \le k$. Consider a $(g,h)$-oriented edge-coloring of $G$ using $k$ colors. Giving each of these $k$ color classes a $(g,h)$-orientation, we can combine these orientations to give a $(k g, k h)$-orientation of $G$. Conversely, suppose that $G$ has a $(k g, k h)$-orientation $D$. We wish to color the arcs of $D$ using $k$ colors so that in each color class, every vertex $v$ of $D$ has indegree at most $g(v)$ and outdegree at most $h(v)$; this would imply that $\chi_{(g,h)}'(G) \le k$. Construct an auxiliary bipartite multigraph $H$ as follows. We have two parts $X$ and $Y$, and for each vertex $v$ of $G$ we put $\min \{\lceil d_G(v)/k \rceil, h(v)\}$ copies of $v$ in $X$ and $\min \{\lceil d_G(v)/k \rceil, g(v)\}$ copies of $v$ in $Y$. For each arc of $D$ with tail $u$ and head $v$, we put an edge in $H$ between some copy of $u$ in $X$ and some copy of $v$ in $Y$. We do this in such a way that every vertex of $H$ has degree at most $k$, which is possible because every vertex $v$ of $D$ has indegree at most $\min \{d_G(v), k \cdot g(v)\}$ and outdegree at most $\min \{d_G(v), k \cdot h(v)\}$. By K\"onig's Theorem \ref{Konig}, $H$ has proper $k$-edge-coloring. Merging the copies of each vertex of $D$ back to a single vertex, this proper $k$-edge-coloring of $H$ gives a desired coloring of the arcs of $D$.
\end{proof}

Theorem \ref{oriented-Goldberg-Seymour} now follows quickly from combining Lemma \ref{orientation-theorem} and Lemma \ref{coloring-lemma}.

\orientedgoldbergseymour*
\begin{proof}
We have already explained that $\chi_{(g,h)}'(G) \ge \max \left\{ \Delta_{(g,h)}(G), \mathcal{W}_{(g,h)}(G) \right\}$. We also already observed that the theorem holds when $\chi_{(g,h)}'(G) = 1$, by Lemma \ref{orientation-theorem}. Now suppose that $\chi_{(g,h)}'(G) = k \ge 2$. By Lemma \ref{coloring-lemma}, $G$ is not $((k-1)g,(k-1)h)$-orientable. By Lemma \ref{orientation-theorem}, this means that either $d(v) \ge (k-1)g(v) + (k-1)h(v) + 1$ for some $v \in V(G)$, or $e(S) \ge \min\{(k-1)g(S), (k-1)h(S)\} + 1$ for some $S \subseteq V(G)$ with $g(S), h(S) \ge 1$. In the former case, we find that
\begin{align*}
    \chi_{(g,h)}'(G) = k \le \frac{d(v)}{g(v)+h(v)} + \left( 1 - \frac{1}{g(v)+h(v)} \right) < \Delta_{(g,h)}(G) + 1,
\end{align*}
so $\chi_{(g,h)}'(G) \le \Delta_{(g,h)}(G)$. In the latter case, we find that
\begin{align*}
    \chi_{(g,h)}'(G) = k \le \frac{e(S)}{\min\{g(S),h(S)\}} + \left( 1 - \frac{1}{\min\{g(S),h(S)\}} \right) < \mathcal{W}_{(g,h)}(G) + 1,
\end{align*}
so $\chi_{(g,h)}'(G) \le \mathcal{W}_{(g,h)}(G)$. Therefore $\chi_{(g,h)}'(G) \le \max \left\{ \Delta_{(g,h)}(G), \mathcal{W}_{(g,h)}(G) \right\}$.
\end{proof}

We remark that the $(g,\infty)$-orientable subgraphs of a multigraph $G$ form the independent sets of a matroid (see \cite{GaWe}), and that Theorem \ref{oriented-Goldberg-Seymour} in this setting specializes to $\chi_{(g,\infty)}'(G) = \mathcal{W}_{(g,\infty)}(G)$, which is a standard application of Edmonds' Matroid Partition Theorem \cite{Ed}.

\section{Applications}

\subsection{On pseudoarboricity} \label{on-pseudoarboricity}
One immediate application of Theorem \ref{oriented-Goldberg-Seymour} is solving the $f$-coloring problem for pseudoforests. A \textit{pseudoforest} is a multigraph such that every connected component has at most one cycle (possibly a loop). The \textit{pseudoarboricity} $pa(G)$ of a multigraph $G$ is the minimum number of colors needed in an edge-coloring of $G$ such that each color class is a pseudoforest. It is well-known that pseudoforests are exactly the $(1,\infty)$-orientable multigraphs (see \cite{GaWe}). Applying Theorem \ref{oriented-Goldberg-Seymour} with $(g,h) = (1,\infty)$ implies Hakimi's theorem \cite{Ha1} that the pseudoarboricity of every multigraph $G$ is given by
\begin{align*}
    pa(G) = \max_{S \subseteq V(G), |S| \ge 1} \left\lceil \frac{e(S)}{|S|} \right\rceil.
\end{align*}

We observe a more general result. For a multigraph $G$ and a function $f : V(G) \rightarrow \mathbb{N} \backslash \{0,1\}$, we call $G$ a \textit{degree-$f$ pseudoforest} if it is a pseudoforest such that every vertex $v$ of $G$ has degree at most $f(v)$. The \textit{degree-$f$ pseudoarboricity} $pa_f(G)$ of a multigraph $G$ is the minimum number of colors needed in an edge-coloring of $G$ such that each color class is a degree-$f$ pseudoforest.  

\begin{obs} \label{observe0}
The degree-$f$ pseudoforests are exactly the $(1,f-1)$-orientable multigraphs.
\end{obs}

To see this, clearly every $(1, f-1)$-orientable multigraph is a degree-$f$ subgraph, and since this multigraph is necessarily $(1, \infty)$-orientable, it is also a pseudoforest. Conversely, given a degree-$f$ pseudoforest, start by orienting the edges of the cycles to form directed cycles. The unoriented edges now form disjoint trees. Root these trees at some vertex; if the tree is part of a cyclic component of the pseudoforest, then root that tree at the unique vertex in the cycle; otherwise, root the tree at a leaf. Now we just orient the edges of each tree away from the root of that tree, and the result will be a $(1, f-1)$-orientation of the pseudoforest.

Applying Theorem \ref{oriented-Goldberg-Seymour} with $(g,h) = (1,f-1)$ gives the following Goldberg-Seymour result for the degree-$f$ pseudoarboricity of a multigraph.

\begin{thm}\label{f-pseudoarboricity}
For every multigraph $G$ and function $f : V(G) \rightarrow \mathbb{N} \backslash \{0,1\}$, we have
\begin{align*}
    pa_f(G) = \max \left\{ \Delta_f(G), pa(G) \right\}.
\end{align*}
\end{thm}

This solves the $f$-coloring problem for pseudoforests. Note that once we start to consider the case when $f$ takes value $1$ on at least two vertices, degree-$f$ pseudoforests can no longer be characterized as $(1,f-1)$-orientable multigraphs. In particular, we run into the much more difficult problem of proper edge-colorings once $f=1$.

\subsection{On arboricity} \label{on-arboricity}

Now we discuss the $f$-coloring problem for forests. Analogous to pseudoarboricity, the \textit{arboricity} $a(G)$ of a loopless multigraph $G$ is the minimum number of colors needed in an edge-coloring of $G$ such that each color class is a forest. A celebrated theorem of Nash-Williams \cite{Na} states that the arboricity of a loopless multigraph $G$ is given by
\begin{align*}
    a(G) = \max_{S \subseteq V(G), |S| \ge 2} \left\lceil \frac{e(S)}{|S|-1} \right\rceil,
\end{align*}
analogous to Hakimi's formula for pseudoarboricity. For a loopless multigraph $G$ and a function $f : V(G) \rightarrow \mathbb{N} \backslash \{0,1\}$, we will say that $G$ is a \textit{degree-$f$ forest} if it is a forest such that every vertex $v$ of $G$ has degree at most $f(v)$. The \textit{degree-$f$ arboricity} $a_f(G)$ is the minimum number of colors needed in an edge-coloring of $G$ such that each color class is a degree-$f$ forest. Compared to the degree-$f$ pseudoarboricity, the degree-$f$ arboricity of a loopless multigraph appears far more difficult to determine. In particular, the case $f=2$ is the well-known problem of linear arboricity, which we will return to soon. First we will show what we can establish for general $f$.

As usual, there is a maximum degree and a maximum density lower bound for the degree-$f$ arboricity: $a_f(G) \ge \Delta_f(G)$ and $a_f(G) \ge a(G)$. It would be interesting to study the extent to which $a_f(G)$ is determined by $\Delta_f(G)$ and $a(G)$. Here, we obtain an upper bound for $a_f(G)$ using Theorem \ref{f-pseudoarboricity} on the degree-$f$ pseudoarboricity $pa_f(G)$ as well as the following simple observation.

\begin{obs}
Every loopless degree-$2f$ pseudoforest can be edge-colored into 2 degree-$f$ forests.
\end{obs}

To see this, start by two-coloring the edges in every cycle of the pseudoforest so that no cycle is monochromatic. This partial coloring is valid by the assumption that $f(v) \ge 2$ for all $v \in V(G)$. The uncolored edges now form disjoint trees. Root these trees at some vertex; if the tree is part of a cyclic component of the pseudoforest, then root the tree at the unique vertex in the cycle; otherwise, root the tree at an arbitrary vertex. Now we greedily color the edges of each tree in a breadth-first search ordering starting at the root, doing this in such a way to satisfy the degree constraint at each vertex for each of the two color classes. In the end, the degree constraints will be satisfied, and due to the first step there will not be any monochromatic cycle, proving the observation. 

This implies that $a_f(G) \le 2pa_{2f}(G)$. Applying Theorem \ref{f-pseudoarboricity} on degree-$f$ pseudoarboricity then gives us the following upper bound for $a_f(G)$.

\begin{thm}\label{f-arboricity}
For every loopless multigraph $G$ and function $f : V(G) \rightarrow \mathbb{N} \backslash \{0,1\}$, we have
\begin{align*}
    a_f(G) \le 2pa_{2f}(G) = \max \left\{2\Delta_{2f}(G), 2pa(G) \right\} \le \max \left\{\Delta_f(G)+1, 2 pa(G)\right\}.
\end{align*}
\end{thm}

Note that $a_f(G) = a(G) = 2 pa(G)$ when $G$ only consists of an even number of parallel edges between two vertices, so this upper bound is sometimes tight, although it is still quite far from the lower bound $a_f(G) \ge \max \left\{\Delta_f(G), a(G)\right\}$.

\subsection{On linear arboricity} \label{on-linear-arboricity}

We now focus on the most studied case, $f=2$, of the $f$-coloring problem for forests. A degree-$2$ forest is known as a \textit{linear forest} (as it is a union of disjoint paths), and the degree-$2$ arboricity of a loopless multigraph $G$ is known as the \textit{linear arboricity} $la(G)$ of $G$, introduced by Harary \cite{Ha2}. We again observe the easy lower bound $la(G) \ge \lceil \Delta(G)/2 \rceil$. However, by Nash-Williams' above formula for arboricity, a simple graph $G$ can have arboricity $a(G)$ as high as $\left\lceil (\Delta(G)+1)/2 \right\rceil$, which is tight if $G$ is $\Delta(G)$-regular, so $la(G) \ge \left\lceil (\Delta(G)+1)/2 \right\rceil$ for some simple graphs $G$. The Linear Arboricity Conjecture of Akiyama, Exoo, and Harary \cite{AkExHa} states that this is the maximum possible value of the linear arboricity of a simple graph.

\begin{conj}[Linear Arboricity Conjecture]\label{linear-arboricity-conjecture}
For every simple graph $G$, we have $la(G) \le \left\lceil (\Delta(G)+1)/2 \right\rceil$.
\end{conj}

It is known to be NP-hard in general to determine the linear arboricity of a simple graph $G$ at least when $\Delta(G)$ is even \cite{Pe}. The Linear Arboricity Conjecture \ref{linear-arboricity-conjecture} has been proven for a few classes of simple graphs (see, e.g., \cite{AkExHa,Gu1}) as well as for some small values of $\Delta(G)$ (see, e.g., \cite{AkExHa,Gu3}). It has been verified to hold asymptotically as $\Delta(G) \rightarrow \infty$, with Lang and Postle \cite{LaPo} recently proving the best known asymptotic upper bound, $la(G) \le \Delta(G)/2 + O(\Delta(G)^{1/2} \log^4 \Delta(G))$. The best known general upper bound, due to Guldan \cite{Gu2}, is $la(G) \le \lceil 3\Delta(G)/5 \rceil$ if $\Delta(G)$ is even, and $la(G) \le \lceil (3\Delta(G)+2)/5 \rceil$ if $\Delta(G)$ is odd. There is also an analogous conjecture for directed graphs \cite{NaPe} that would imply Conjecture \ref{linear-arboricity-conjecture}. 

Recently, there has been interest in proving the Linear Arboricity Conjecture \ref{linear-arboricity-conjecture} for sparse graphs, commonly taken to be graphs of low degeneracy. A loopless multigraph $G$ is said to be \textit{$k$-degenerate} if every subgraph of $G$ has a vertex of degree at most $k$. The following observation will be useful for us.

\begin{obs} \label{observe}
For every $k$-degenerate loopless multigraph $G$, we have $pa(G) \le a(G) \le k$.
\end{obs}

To see this, note that $k$-degeneracy is equivalent to $G$ having an acyclic orientation $D$ such that every vertex has indegree at most $k$, constructed by iteratively deleting a vertex $v$ of degree at most $k$ in $G$ and orienting the deleted edges toward $v$ in $D$. By coloring all the incoming arcs at each vertex of $D$ a different color, we obtain an edge-coloring of $G$ into $k$ forests. Thus $a(G) \le k$, and this proves the stated inequalities.

Now, Kainen \cite{Ka1} has proven that $la(G) \le \lceil (\Delta(G)+k-1)/2 \rceil$ if $G$ is a $k$-degenerate simple graph, which implies the Linear Arboricity Conjecture \ref{linear-arboricity-conjecture} when $k=2$. More recently, Basavaraju, Bishnu, Francis, and Pattanayak \cite{BaBiFrPa} proved the conjecture for $3$-degenerate simple graphs. They also proved that $la(G) = \lceil \Delta(G)/2 \rceil$ for all $2$-degenerate simple graphs $G$ when $\Delta(G) \ge 5$. Chen, Hao, and Yu \cite{ChHa} recently announced that they proved the Linear Arboricity Conjecture \ref{linear-arboricity-conjecture} for $k$-degenerate simple graphs $G$ when $\Delta(G) \ge 2k^2 - 2k$. Using Theorem \ref{f-arboricity} on the degree-$f$ arboricity, we will improve their bound to $\Delta(G) \ge 4k-2$. Moreover, our result will apply more generally to when $G$ is a loopless multigraph.

We comment that the linear arboricity of loopless multigraphs has not received as much attention as simple graphs. The main contribution for multigraphs has been one paper of A\"it-djafer \cite{Ai}, who generalized Conjecture \ref{linear-arboricity-conjecture} to $la(G) \le \lceil (\Delta(G)+\mu(G))/2 \rceil$ where $\mu(G)$ is the edge-multiplicity of $G$. She verified this for $\mu(G) \ge \Delta(G)-2$, as well as when $\Delta(G)$ is close to a power of 2 and $\mu(G)$ is close to $\Delta(G)/2$. Note again that by Nash-Williams' formula, the arboricity $a(G)$ of a loopless multigraph $G$ can be as high as $\left\lceil (\Delta(G)+\mu(G))/2 \right\rceil$, so this conjectured upper bound for $la(G)$ would be best possible in terms of $\Delta(G)$ and $\mu(G)$ alone. What we will observe is that if $G$ is a sufficiently sparse loopless multigraph, then we can improve this conjectured upper bound for linear arboricity to $la(G) \le \left\lceil (\Delta(G)+1)/2 \right\rceil$. In particular we will be able to deduce the Linear Arboricity Conjecture \ref{linear-arboricity-conjecture} for such multigraphs $G$. This sparsity phenomenon also occurs with the Goldberg-Seymour Conjecture \ref{Goldberg-Seymour-f} for $f$-colorings.

\begin{thm} \label{linear-arboricity-theorem}
For every loopless multigraph $G$ with $\Delta(G) \ge 4pa(G)-2$, we have 
\begin{align*}
    la(G) \le \left\lceil \frac{\Delta(G)+1}{2} \right\rceil.
\end{align*}
\end{thm}

\begin{proof}
Applying Theorem \ref{f-arboricity} with $f = 2$ shows that every loopless multigraph $G$ satisfies
\begin{align*}
    la(G) \le 2pa_4(G) = \max \left\{ 2 \left\lceil \frac{\Delta(G)}{4} \right\rceil, 2pa(G) \right\}.
\end{align*}
In particular, when $\Delta(G) \ge 4pa(G)-3$, we have
\begin{align*}
    la(G) \le 
    \begin{cases}
    \left\lceil \Delta(G)/2 \right\rceil &\mbox{if } \Delta(G) \equiv 0 \text{ or } 3 \pmod 4, \\
    \left\lceil \Delta(G)/2 \right\rceil + 1 &\mbox{if } \Delta(G) \equiv 1 \text{ or } 2 \pmod 4.
    \end{cases}
\end{align*}
When $\Delta(G) \equiv 2 \pmod 4$, we have $\left\lceil \Delta(G)/2 \right\rceil + 1 = \left\lceil (\Delta(G)+1)/2 \right\rceil$, so the proof is complete in all cases except $\Delta(G) \equiv 1 \pmod 4$. For this final case, we use a trick due to Guldan \cite{Gu2} of deleting a linear forest from $G$ before applying our theorem. This is done via the following simple claim. (Note that Guldan stated this claim only for regular simple graphs and gave no proof.)

\begin{claim}
Every loopless multigraph $G$ has a linear forest $F$ such that every vertex of degree $\Delta(G)$ in $G$ has degree at least one in $F$.
\end{claim}

\begin{proof}
Applying the same logic as Lemma \ref{coloring-lemma}, the multigraph $G$ can be edge-colored into $\lceil \Delta(G)/2 \rceil$ degree-$2$ subgraphs, which are each unions of disjoint paths and cycles. Letting $F'$ be any one of these degree-$2$ subgraphs, we see that every vertex $v$ of degree $\Delta(G)$ in $G$ must have degree at least one in $F'$. Then letting $F$ be obtained from $F'$ by removing one edge from each cycle in $F'$, we obtain our desired linear forest $F$ in $G$.
\end{proof}

We may now complete the proof. Let $G$ be a multigraph with $\Delta(G) \equiv 1 \pmod 4$, let $F$ be a linear forest from the above claim, and let $G' = G - E(F)$. Then $G = G' \cup F$ and $\Delta(G') \le \Delta(G)-1$, so we have 
\begin{align*}
    la(G) &\le la(G') + 1 \le 2pa_4(G') + 1 = \max \left\{ 2\left\lceil \frac{\Delta(G')}{4} \right\rceil + 1, 2pa(G') + 1\right\} \\
    &\le \max \left\{ 2 \cdot \frac{\Delta(G)-1}{4} + 1, 2pa(G) + 1 \right\} = \max \left\{ \frac{\Delta(G)+1}{2}, 2pa(G) + 1 \right\}.
\end{align*}
When $\Delta(G) \ge 4pa(G) - 2$, this shows that $la(G) \le \lceil (\Delta(G)+1)/2 \rceil$ as required.
\end{proof}

\begin{cor} \label{linear-arboricity-corollary}
The Linear Arboricity Conjecture \ref{linear-arboricity-conjecture} holds for all $k$-degenerate loopless multigraphs $G$ with $\Delta(G) \ge 4k-2$.
\end{cor}

\begin{proof}
This is immediate from Observation \ref{observe} and Theorem \ref{linear-arboricity-theorem}.
\end{proof}

Finally, we note that Theorem \ref{f-arboricity} on the degree-$f$ arboricity also improves an old bound of Caro and Roditty \cite{CaRo}. Generalizing the linear arboricity upper bound of Kainen \cite{Ka1} stated above, Caro and Roditty proved that every $k$-degenerate simple graph $G$ satisfies $a_t(G) \le \left\lceil \frac{\Delta(G) + (t-1)k - 1}{t} \right\rceil$, where $f=t \ge 2$ is taken to be constant. Applying our Theorem \ref{f-arboricity} with $f=t$, we obtain the following upper bound for $a_t(G)$ which improves Caro and Roditty's bound at least when $\Delta(G) \ge (k+1)t+2$. Our result also applies more generally to when $G$ is a loopless multigraph.

\begin{cor}\label{t-arboricity}
For every loopless multigraph $G$, we have
\begin{align*}
    a_t(G) \le \max \left\{ \left\lceil \frac{\Delta(G)}{t} \right\rceil + 1, 2pa(G) \right\}.
\end{align*}
\end{cor}

\subsection{On the f-chromatic index} \label{on-f-chromatic-index}
As a last application, we indicate elementary results on the $f$-chromatic index $\chi_f'(G)$ that we can immediately obtain from Theorem \ref{oriented-Goldberg-Seymour}. First we observe that we can easily derive the following exact result on the $f$-chromatic index due to Hakimi and Kariv \cite{HaKa}.

\begin{cor} \label{bipartite-even}
Let $G$ be a multigraph and let $f : V(G) \rightarrow \mathbb{N} \backslash \{0\}$ be a function.
\begin{enumerate}[label={(\roman*)}]
    \item If $G$ is bipartite, then $\chi_f'(G) = \Delta_f(G)$.
    \item If $f(v)$ is even for all $v \in V(G)$, then $\chi_f'(G) = \Delta_f(G)$.
\end{enumerate}
\end{cor}

For a sketch, to prove (i), if $X$ and $Y$ are the two parts of $G$, then we apply Theorem \ref{oriented-Goldberg-Seymour} with the valid pair of functions $(g,h)$ defined by $(g(v), h(v)) = (0, f(v))$ if $v \in X$, and $(g(v), h(v)) = (f(v), 0)$ if $v \in Y$. To prove (ii), we apply Theorem \ref{oriented-Goldberg-Seymour} with the valid pair $(g,h)$ defined by $(g(v), h(v)) = (f(v)/2, f(v)/2)$ for all $v \in V(G)$. In both cases, it is easy to check that the weighted maximum degree $\Delta_{(g,h)}(G) = \Delta_f(G)$ is at least the weighted maximum density $\mathcal{W}_{(g,h)}(G) = \mathcal{W}_f(G)$, so that Corollary \ref{bipartite-even} holds. Although Hakimi and Kariv's proof of Corollary \ref{bipartite-even} is not too different from this approach, our proof sketch provides a more unified view of this result.

On the other hand, for general multigraphs $G$ and general functions $f : V(G) \rightarrow \mathbb{N} \backslash \{0\}$ the best kind of result we can hope for is the Goldberg-Seymour Conjecture \ref{Goldberg-Seymour-f} for $f$-colorings,
\begin{align*}
    \chi_f'(G) \le \max \left\{\Delta_f(G) + 1, \max_{S \subseteq V(G), |S| \ge 2} \left\lceil \frac{e(S)}{\lfloor f(S)/2 \rfloor} \right\rceil \right\}.
\end{align*}
If we assume that $f(v) \ge 2$ for all $v \in V(G)$, then we may obtain an approximation of the conjecture by taking $g(v) = \lfloor f(v)/2 \rfloor$ and $h(v) = \lceil f(v)/2 \rceil$ for all $v \in V(G)$. Every $(g,h)$-orientable subgraph is indeed a degree-$f$ subgraph, so Theorem \ref{oriented-Goldberg-Seymour} implies the following.

\begin{thm} \label{odd-cor}
For every multigraph $G$ and function $f : V(G) \rightarrow \mathbb{N} \backslash \{0,1\}$, we have
\begin{align*}
    \chi_f'(G) \le \max \left\{\Delta_f(G), \max_{S \subseteq V(G), |S| \ge 2} \left\lceil \frac{e(S)}{\sum_{v \in S} \lfloor f(v)/2 \rfloor} \right\rceil \right\}.
\end{align*}
\end{thm}
This result can be viewed as an ``averaged" version of Hakimi and Kariv's Theorem 2 in \cite{HaKa}, and it is an improvement in most cases. As the values of $f$ get larger, this approximation of Conjecture \ref{Goldberg-Seymour-f} gets better, independently of the size of the multigraph: If $f(v) \ge k$ for all $v \in V(G)$, then the factor of approximation is about $k/(k-1)$. Of course, it does not come close to approximations achieved using the more precise tools of alternating trails and Tashkinov trees (see \cite{NaNiSa, St}).

\section{List coloring} \label{list}
Given a multigraph $G$, a \textit{list assignment} $L$ for $E(G)$ is an assignment of a list $L(e)$ of distinct colors to each edge $e$ of $G$. Given a list assignment $L$ for $E(G)$, an \textit{$L$-coloring} of $G$ is an edge-coloring such that the color of each edge $e$ lies in $L(e)$. A \textit{proper $L$-coloring} of $G$ is an $L$-coloring such that each color class is a matching. The \textit{list chromatic index} $\chi_\ell'(G)$ of a loopless multigraph $G$ is the minimum $k$ such that for every list assignment $L$ with $|L(e)| \ge k$ for all $e \in E(G)$, there exists a proper $L$-coloring of $G$. Notice that $\chi_\ell'(G) \ge \chi'(G)$ because there cannot exist a proper $L$-coloring of $G$ for the list assignment $L(e) = \{1, \ldots, \chi'(G)-1\}$, $e \in E(G)$. The infamous List Coloring Conjecture asserts that this inequality should in fact be an equality.

\begin{conj}[List Coloring Conjecture] \label{list-coloring-conjecture}
For every loopless multigraph $G$, we have $\chi_\ell'(G) = \chi'(G)$.
\end{conj}

This conjecture was first stated explicitly by Bollob\'as and Harris \cite{BoHa} but was also suggested by various authors before them. It has been verified for only a few classes of multigraphs, although Kahn \cite{Ka} has confirmed that it holds asymptotically as $\Delta(G) \rightarrow \infty$ at least when $G$ is a simple graph. Our interest in the conjecture is the well-known special case when $G$ is a bipartite multigraph, which was proven in surprising fashion by Galvin \cite{Ga}. 

\begin{thm}[Galvin] \label{Galvin}
For every bipartite multigraph $G$, we have $\chi_\ell'(G) = \chi'(G)$.
\end{thm}

We can similarly define the list coloring analogue $\chi_{(g,h),\ell}'(G)$ of the $(g,h)$-oriented chromatic index $\chi_{(g,h)}'(G)$, namely as the minimum $k$ such that for every list assignment $L$ for $E(G)$ with $|L(e)| \ge k$ for all $e \in E(G)$, there exists an $L$-coloring of $G$ where every color class is a $(g,h)$-orientable subgraph of $G$. Using Lemma \ref{coloring-lemma} and Galvin's Theorem \ref{Galvin}, we can immediately prove the analogue of the List Coloring Conjecture \ref{list-coloring-conjecture} for $(g,h)$-oriented-colorings.

\begin{restatable}{thm}{listversion}\label{list-version}
For every multigraph $G$ and valid pair of functions $(g,h)$, we have
\begin{align*}
    \chi'_{(g,h),\ell}(G) = \chi'_{(g,h)}(G).
\end{align*}
\end{restatable}

\begin{proof}
Suppose that $\chi'_{(g,h)}(G) = k \ge 1$. Let $L$ be any list assignment for $E(G)$ with $|L(e)| \ge k$ for all $e \in E(G)$. By Lemma \ref{coloring-lemma}, $G$ has a $(kg, kh)$-orientation $D$. Let $H$ be the same auxiliary bipartite multigraph associated to $D$ as in the proof of Lemma \ref{coloring-lemma}, and assign the lists of $L$ to the corresponding edges in $H$. By construction, $H$ has maximum degree at most $k$, so by Galvin's Theorem \ref{Galvin}, $H$ has a proper $L$-coloring. Merging copies of each vertex of $D$ back to a single vertex, this proper $L$-coloring of $H$ gives an $L$-coloring of $G$ into $(g,h)$-orientable subgraphs, as required.
\end{proof}

We observe that Theorem \ref{list-version} implies list coloring results for some of the other edge-coloring parameters we studied above. Let $pa_{f,\ell}(G)$ denote the list coloring analogue of the degree-$f$ pseudoarboricity $pa_f(G)$. Then Theorem \ref{list-version} together with Observation \ref{observe0} implies that $pa_{f,\ell}(G) = pa_f(G)$ if $f(v) \ge 2$ for all $v \in V(G)$. Next, let $\chi_{f,\ell}'(G)$ denote the list coloring analogue of the $f$-chromatic index $\chi_f'(G)$. Then Theorem \ref{list-version} and our comments on Corollary \ref{bipartite-even} imply that $\chi_{f,\ell}'(G) = \chi_f'(G)$ if $G$ is bipartite or if $f$ takes only even values, since we saw in these cases that $\chi_f'(G) = \chi_{(g,h)}'(G)$ for an appropriate choice of valid functions $(g,h)$. In particular, Galvin's Theorem \ref{Galvin} is itself a special case of Theorem \ref{list-version}. Also note that the upper bound on $\chi_{f}'(G)$ in Corollary \ref{odd-cor} applies as well to $\chi_{f,\ell}'(G)$. The list coloring analogue of the degree-$f$ arboricity $a_f(G)$ appears harder to study by our approach, so we only remark that a result of Seymour \cite{Se1} states that list arboricity equals arboricity (the case $f = \infty$), and that the above mentioned asymptotic upper bound for linear arboricity by Lang and Postle \cite{LaPo} was in fact proven for list linear arboricity (the case $f=2$).

\section*{Acknowledgments}
I would like to thank Penny Haxell for the much helpful editing of this work, encouraging me to finish writing it, and being the source of many ideas that led to these results.

\appendix 
\section{Proof of Lemma \ref{orientation-theorem}} \label{appendix}
\orientationtheorem*

\begin{proof}
If $G$ is $(g,h)$-orientable, then conditions (1) and (2) clearly hold as explained in the introduction. Conversely, suppose that $G$ is not $(g,h)$-orientable. If condition (1) is violated for some $v \in V(G)$, then we are done, so assume that $d_G(v) \le g(v) + h(v)$ for all $v \in V(G)$. Let $D$ be a maximum partial $(g,h)$-orientation of $G$. By assumption, some edge $e \in E(G)$ is not oriented. 

First assume that $e$ is not a loop. Let $u$ and $v$ be the end-vertices of $e$. Since $e$ is not oriented, we have
\begin{align*}
    d_D^-(u) + d_D^+(u) \le d_G(u) - 1 \le g(u) + h(u) - 1,
\end{align*}
which implies we have either $d_D^-(u) < g(u)$ or $d_D^+(u) < h(u)$. Likewise, we have either $d_D^-(v) < g(v)$ or $d_D^+(v) < h(v)$. We cannot have both $d_D^-(u) < g(u)$ and $d_D^+(v) < h(v)$, as otherwise we can orient $e$ from $v$ to $u$ and obtain a larger $(g,h)$-orientation of $G$, contradicting the maximality of $D$. Similarly, we cannot have both $d_D^-(v) < g(v)$ and $d_D^+(u) < h(u)$.

Suppose first that $d_D^-(u) = g(u)$ and $d_D^-(v) = g(v)$, so that $d_D^+(u) < h(u)$ and $d_D^+(v) < h(v)$. Let $S$ be the set of all vertices $w$ of $G$ such that there is a directed path in $D$ from $w$ to either $u$ or $v$. We claim that $d_D^-(w) = g(w)$ for all $w \in S$. By assumption this holds for $w = u$ and $w = v$. Suppose on the contrary that $d_D^-(w) < g(w)$ for some $w \in S$. Let $P$ be a directed path from $w$ to one of $u$ or $v$, say to $v$, and let $D'$ be the partial orientation obtained from $D$ by reversing all the arcs along $P$. Then $D'$ is still a partial $(g,h)$-orientation because $d_{D'}^-(w) = d_D^-(w)+1 \le g(w)$, $d_{D'}^+(v) = d_D^+(v)+1 \le h(v)$, and the indegree and outdegree of the other vertices remain the same. But now $d_{D'}^-(v) = d_D^-(v)-1 = g(v)-1$ and $d_{D'}^+(u) = d_D^+(u) \le h(u)-1$, so we can orient $e$ from $u$ to $v$ and obtain a larger $(g,h)$-orientation of $G$. This contradicts the maximality of $D$ and proves the claim. Now, by definition of $S$, every arc whose head lies in $S$ also has its tail in $S$. Having proven that $d_D^-(w) = g(w)$ for all $w \in S$, by summing the indegrees of vertices in $S$ and including the unoriented edge $e$ we find that
\begin{align*}
    e_G(S) \ge 1 + \sum_{w \in S} d_D^-(w) = 1 + g(S).
\end{align*}
This violates condition (2) as required.

Suppose instead that $d_D^+(u) = h(u)$ and $d_D^+(v) = h(v)$, so that $d_D^-(u) < g(u)$ and $d_D^-(v) < g(v)$. Then let $S$ be the set of all vertices $w$ of $G$ such that there is a directed path in $D$ from either $u$ or $v$ to $w$. A similar path reversal argument shows that $d_D^+(w) = h(w)$ for all $w \in S$, and thus $e_G(S) \ge 1 + \sum_{w \in S} d_D^+(w) = 1 + h(S)$,
again violating condition (2).

Finally, assume that $e$ is a loop at the vertex $v$. Since $e$ is not oriented, we have
\begin{align*}
    d_D^-(v) + d_D^+(v) \le d_G(v) - 2 \le g(v) + h(v) - 2.
\end{align*}
If $d_D^-(v) \le g(v) - 1$ and $d_D^+(v) \le h(v) - 1$, then we can orient $e$ arbitrarily and contradict the maximality of $D$. If $d_D^-(v) = g(v)$ and $d_D^+(v) \le h(v) - 2$, then the same arguments as above let us find a vertex set $S \subseteq V(G)$ violating condition (2). The same applies if $d_D^-(v) \le g(v) - 2$ and $d_D^+(v) = h(v)$. This completes the proof.
\end{proof}

\end{document}